\documentclass[a4paper, 12pt]{article}

\usepackage[utf8]{inputenc}
\usepackage[T1]{fontenc}
\usepackage[a4paper, margin=2.5cm]{geometry}
\usepackage{needspace}

\usepackage{tikz}
\usetikzlibrary{positioning}

\usepackage{libertine} 
\usepackage{inconsolata} 

\usepackage{amsmath, amsthm, amssymb}
\usepackage{graphicx}
\usepackage{enumerate}
\usepackage{authblk}
\usepackage{thmtools}
\usepackage{thm-restate}
\usepackage[colorlinks=true, citecolor=red]{hyperref}
\usepackage{float}
\usepackage{mathtools}

\renewenvironment{abstract}
{\small\vspace{-1em}
\begin{center}
\bfseries\abstractname\vspace{-.5em}\vspace{0pt}
\end{center}
\list{}{
\setlength{\leftmargin}{0.6in}%
\setlength{\rightmargin}{\leftmargin}}%
\item\relax}
{\endlist}

\declaretheorem[name=Theorem, numberwithin=section]{theorem}

\declaretheorem[name=Conjecture, sibling=theorem]{conjecture}

\declaretheorem[name=Claim, sibling=theorem]{claim}

\declaretheorem[name=Remark, style=remark, sibling=theorem]{remark}

\def\cqedsymbol{\ifmmode$\lrcorner$\else{\unskip\nobreak\hfil
\penalty50\hskip1em\null\nobreak\hfil$\lrcorner$
\parfillskip=0pt\finalhyphendemerits=0\endgraf}\fi}

\let\leq\leqslant
\let\geq\geqslant

\title{On cuts of small chromatic number in sparse graphs}

\author[1]{Guillaume Aubian}
\author[2]{Marthe Bonamy}
\author[2]{Romain Bourneuf}
\author[2]{Oscar Fontaine}
\author[,3]{Lucas Picasarri-Arrieta\footnote{Research supported by JST as part of ASPIRE, Grant Number JPMJAP2302.}}
\affil[1]{Université Paris-Panthéon-Assas, CRED Paris, France.}
\affil[2]{CNRS, LaBRI, Université de Bordeaux, Bordeaux, France.}
\affil[3]{National Institute of Informatics, Tokyo, Japan.}
\date{\today}

\begin{document}

\maketitle

\begin{abstract}
For a given integer $k$, let $\ell_k$ denote the supremum $\ell$ such that every sufficiently large graph $G$ with average degree less than $2\ell$ admits a separator $X \subseteq V(G)$ for which $\chi(G[X]) < k$. 
Motivated by the values of $\ell_1$, $\ell_2$ and $\ell_3$, a natural conjecture suggests that $\ell_k = k$ for all $k$. 
We prove that this conjecture fails dramatically: asymptotically, the trivial lower bound $\ell_k \geq \tfrac{k}{2}$ is tight. More precisely, we prove that for every $\varepsilon>0$ and all sufficiently large $k$, we have $\ell_k \leq (1+\varepsilon)\tfrac{k}{2}$.  
\end{abstract}


\section{Introduction}\label{sec:intro}

For a given integer $k$, we define $\ell_k$ as the supremum $\ell$ such that every sufficiently large\footnote{As a function of $k$ and $\ell$.} graph~$G$ with average degree less than $2\ell$ contains a set $X \subseteq V(G)$ with the properties that~$G \setminus X$ is disconnected and $\chi(G[X])<k$.

Observe first that any graph with average degree less than $k$ and order at least~$k+1$ contains a vertex of degree at most $k-1$, whose neighbourhood is thus a separator of chromatic number less than $k$.
Conversely, for any $n \geq k$, one can construct an $n$-vertex graph with no such set~$X$ by taking a clique on $k$ vertices and joining it completely to an independent set of size $n-k$, see~\autoref{fig:ell_leq_k} for an illustration. This graph has exactly $kn - \tfrac{k(k+1)}{2}$ edges and hence average degree less than $2k$, yet every separator must include the entire clique, which has chromatic number $k$. 
This shows that $\ell_k$ is well-defined and satisfies
\[
\frac{k}{2} \;\leq\; \ell_k \;\leq\; k.
\]

\begin{figure}
    \centering
    \begin{tikzpicture}
      \usetikzlibrary{decorations.pathreplacing} 
      \tikzset{vertex/.style = {circle, black!80, fill, inner sep=1.5pt, outer sep=1pt}}

      \foreach \i in {1,...,5} {
        \node[vertex] (c\i) at (18+72*\i:1) {};
      }
      
      \node[vertex] (s1) at (3.5,0.6) {};
      \node[draw=none,fill=none,inner sep=0pt] (dots1) at (3.5,0.2) {$\vdots$};
      \node[vertex] (s2) at (3.5,-0.4) {};

      \foreach \i in {1,...,5}
        \foreach \j in {s1,s2}{
          \draw[line width=1.6pt, white] (c\i) -- (\j);
          \draw[gray] (c\i) -- (\j);
          }
          
      \foreach \i in {1,...,4}{
        \pgfmathtruncatemacro{\j}{\i + 1}
        \foreach \k in {\j,...,5}{
          \draw[line width=2pt, white] (c\i) -- (c\k);
          \draw[semithick] (c\i) -- (c\k);
          }
      }

      \draw[decorate,decoration={brace,amplitude=5pt,mirror}] 
        (-1.3,1.2) -- (-1.3,-1) 
        node[midway,xshift=-0.8cm, draw=none, fill=none] {$k$};

      \draw[decorate,decoration={brace,amplitude=5pt}] 
        (3.9,0.8) -- (3.9,-0.6) 
        node[midway,xshift=1.1cm, draw=none, fill=none] {$n-k$};
    \end{tikzpicture}

    \caption{An $n$-vertex graph of average degree less than $2k$ in which every separator has chromatic number at least $k$.}
    \label{fig:ell_leq_k}
\end{figure}
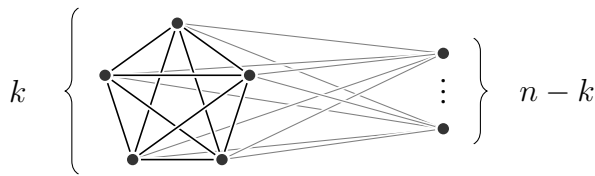

Having established general bounds, let us now consider the small cases. It is folklore that~$\ell_1 = 1$: indeed, whenever $n \geq 2$, every $n$-vertex graph with fewer than $n-1$ edges is disconnected, while connected graphs with average degree $2$ certainly exist (e.g. cycles).

The case $k=2$ was resolved by Chen and Yu \cite{CY02}, confirming a conjecture of Caro with a very elegant inductive proof.

\begin{theorem}[\cite{CY02}]\label{th:k=2}
Any graph on $n$ vertices with fewer than $2n-3$ edges admits a stable cut, while some graphs with exactly $2n-3$ edges do not. In particular, $\ell_2=2$.
\end{theorem}

The extremal graphs with $2n-3$ edges and no stable cut were first characterized by Le and Pfender \cite{LP13}, although their proof contained a gap later filled by Rauch and Rautenbach~\cite{RR24}.

The next case, $k=3$, was investigated by Bogdanov, Neustroeva, Sokolov, Volostnov, Russkin, and Voronov \cite{BNSVRV25}, who formulated the following conjecture.

\begin{conjecture}[\cite{BNSVRV25}]\label{conj:k=3}
Any $n$-vertex graph with fewer than $3n-6$ edges admits a bipartite cut, while some graphs with $3n-6$ edges do not. In particular, $\ell_3=3$.
\end{conjecture}

The extremal examples here are provided by $3$-trees. In fact, an even stronger conjecture predates this one, namely that under the same conditions one can always find a cut inducing a forest \cite{CRR25}. Partial progress was obtained in the same paper by Chernyshev, Rauch and Rautenbach, who proved that every $n$-vertex graph with fewer than $\tfrac{11}{5}n-\tfrac{18}{5}$ edges admits a forest cut. This bound was subsequently improved to $\tfrac{9}{4}n-\tfrac{15}{4}$ by Botler, Couto, Fernandes, de Figueiredo, Gómez, dos Santos and Sato \cite{BCFdFGdSS25}, and then to $\tfrac{19}{8}n-\tfrac{28}{8}$ by Bogdanov et al.~\cite{BNSVRV25}. In the same work, the authors also established a bound of $\tfrac{80}{31}n-\tfrac{134}{31}$ for bipartite cuts. Related results were later obtained by Cheng, Tang and Zhan \cite{CTZ26}.

Taken together, these cases naturally suggest a bold generalization:

\begin{conjecture}\label{conj:bold}
For every integer $k$ and every graph $G$ on at least $k$ vertices, if
\[
|E(G)| < k |V(G)| - \tfrac{k(k+1)}{2},
\]
then $G$ admits a cut $X$ with $\chi(G[X])<k$. 
In particular, $\ell_k = k$.
\end{conjecture}

The main purpose of this note is to show that \autoref{conj:bold} is in fact far from correct.

\begin{restatable}{theorem}{main}\label{th:main}
For any $\varepsilon>0$ and all sufficiently large $k$, we have
\[
\ell_k \leq (1+\varepsilon)\,\tfrac{k}{2}.
\]
\end{restatable}

While this does not fully determine $\ell_k$, it provides an essentially sharp asymptotic estimate when combined with the lower bound:
\[
    \tfrac{k}{2} \leq \ell_k \leq (1+o(1))\tfrac{k}{2}
\]

Thus we arrive at the following conclusion.

\begin{theorem}\label{th:main2}
    As $k $ grows large, we have $\ell_k \sim \frac{k}2$.
\end{theorem}

In fact, we prove a stronger statement:

\begin{theorem}\label{th:clique-separator}
    For every integer $k$, there exist arbitrarily large graphs with average degree~${(1 + o(1))k}$ in which every separator contains a clique of size $k$.
\end{theorem}

This construction is interesting in its own right, and appeared in~\cite{bessy2025sparse} where it was used to establish lower bounds on the smallest maximum degree of a cut (instead of its chromatic number). Since the chromatic number is always at most the degeneracy plus one,~\autoref{th:main} also rules out the strengthening of~\autoref{conj:bold} where $\chi(G[X])<k$ is replaced by the requirement that $G[X]$ be $(k-2)$-degenerate. 
This would have tied in neatly with the already studied cases:
\begin{itemize}
    \item the only \text{-}1-degenerate graph is the empty one,
    \item a $0$-degenerate graph is stable,
    \item a $1$-degenerate graph is a forest.
\end{itemize}

Thus, for $k=1,2$, this is equivalent to the standard definition using chromatic number, while with $k=3$ we retrieve the well-studied notion of forest-cuts. We note that the condition $\chi(G[X])<k$ seems easier to work with when attempting to obtain positive results, as it is compatible with identifying a stable set into a single vertex\footnote{For any smallest graph $G$ with $\ell |V(G)|-|E(G)|>c$ and no cut of chromatic number less than $k$, we obtain that every subset $X$ of vertices is either a clique or satisfies $\ell |X|-|E(G[X])|>\ell t - \frac{t(t-1)}2$, where $t=\chi(G[X])$.}.

\section{Proofs}\label{sec:proofs}

In a bipartite graph $G = (A \cup B, E)$, a \emph{bi-hole of size $k$} is a pair $(A', B')$ with $\min\{|A'|,|B'|\} = k$, $A' \subseteq A$, $B' \subseteq B$, such that there is no edge between $A'$ and $B'$.
In some sense, the size of a largest bi-hole in a bipartite graph corresponds to the ``bipartite independence number'' of that graph.
Axenovich, Sereni, Snyder, and Weber \cite{ASSW21} studied the following question: what is the largest integer $f(n, \Delta)$ such that every $n \times n$ bipartite graph $G = (A \cup B, E)$ with~$\deg(a) \leq \Delta$ for every vertex $a \in A$ contains a bi-hole of size $f(n, \Delta)$?
They proved that the asymptotic behaviour of the function $f(n, \Delta)$ is $\Theta\left(\frac{\ln\Delta}{\Delta} \cdot n\right)$. We make use of the following upper bound.

\begin{theorem}[\cite{ASSW21}] \label{thm:ASSW}
    Let $\Delta \geq 27$ be an integer and $n \geq \frac{\Delta}{\ln\Delta}$. Then, there exists an $n \times n$ bipartite graph $G = (A \cup B, E)$ with $\deg(a) \leq \Delta$ for every vertex $a \in A$, which contains no bi-hole of size at least $8 \cdot \frac{\ln\Delta}{\Delta} \cdot n$.
\end{theorem}

Such a graph can be obtained with high probability from a random bipartite graph $G(2n, 2n, \Delta/(4n))$ by restricting one part to $n$ vertices of degree at most $\Delta$ and the other part to any set of $n$ vertices.

We now prove \autoref{th:main}, which we restate for convenience.

\main*

\begin{proof}
    Fix $\varepsilon > 0$, set $\eta \coloneqq \varepsilon/2$ and let $\Delta \geq 27$ be large enough so that $1 - 8 \cdot \frac{\ln\Delta}{\Delta} \geq \frac{1}{1+\eta}$.
    Let $k$ be an integer large enough so that $\eta k \geq 2\Delta$ and $(1+\eta) k \geq \frac{\Delta}{\ln\Delta}$.
    Set $\ell \coloneqq (1+\varepsilon) \frac{k}{2}$.
    To show that $\ell_k \leq \ell$, it suffices to prove that there exist arbitrarily large graphs~$G$ with average degree less than $2\ell$ and where every separator $X \subseteq V(G)$ of $G$ satisfies $\chi(G[X]) \geq k$.

    Set $\alpha \coloneqq \lceil(1+\eta) k\rceil \geq \frac{\Delta}{\ln\Delta}$, and let $\beta \geq 2$ be an integer.
    By \autoref{thm:ASSW}, there exists an~$\alpha \times \alpha$ bipartite graph $H = (A \cup B, E)$ with $\deg(a) \leq \Delta$ for every vertex $a \in A$, which contains no bi-hole of size at least $8 \cdot \frac{\ln\Delta}{\Delta} \cdot \alpha$.
    Consider the graph $G$ whose vertex set is the union of $\beta$ pairwise disjoint sets  $A_1, \ldots, A_\beta$ of $\alpha$ vertices each, and whose edges are exactly such that:
    \begin{itemize}
        \item for every $i \in [1,\beta]$, the graph $G[A_i]$ is a clique, and
        \item for every $i \in [1,\beta-1]$, the semi-induced subgraph $G[A_i, A_{i+1}]$ is isomorphic to $H$, with~$A_i$ mapped to the part $A$ of $H$ and $A_{i+1}$ to the part $B$ of $H$.
    \end{itemize}

    \begin{claim}
        Every separator $X \subseteq V(G)$ of $G$ satisfies $\chi(G[X]) \geq k$.
    \end{claim}

    \begin{proof}
        Consider a set $X \subseteq V(G)$ such that $G \setminus X$ is disconnected.
        Since each~$G[A_i]$ is a clique, there exists an integer $i \in [1,\beta-1]$ such that there is no edge in $G$ between~$A_i \setminus X$ and $A_{i+1} \setminus X$.
        By construction of $G$, this means that $(A_i \setminus X, A_{i+1} \setminus X)$ is a bi-hole in~$G[A_i, A_{i+1}] \cong H$.
        Therefore, by definition of $H$, we have \[\min\{|A_i \setminus X|, |A_{i+1} \setminus X|\} \leq 8 \cdot \frac{\ln\Delta}{\Delta} \cdot \alpha.\]
        Thus, we have \[\max\{|A_i \cap X|, |A_{i+1} \cap X|\} \geq \alpha\left(1-8 \cdot \frac{\ln\Delta}{\Delta}\right).\]
        Since $G[A_i]$ and $G[A_{i+1}]$ are cliques, we deduce \[{\chi(G[X]) \geq \alpha\left(1-8 \cdot \frac{\ln\Delta}{\Delta}\right) \geq (1 + \eta) k \cdot \frac{1}{1+\eta} \geq k}.\qedhere\]
    \end{proof}

    \begin{claim}
         $G$ has average degree less than $2\ell$.
    \end{claim}

    \begin{proof}
        In $H$, every vertex $a \in A$ satisfies $\deg(a) \leq \Delta$, so $H$ has at most $\alpha \Delta$ edges.
        Therefore, \[|E(G)| \leq \beta \left(\frac{\alpha(\alpha-1)}{2} + \alpha \Delta\right).\]
        Moreover $|V(G)|=\beta\alpha$. Thus, the average degree of $G$ is 
        \[
        \frac{2|E(G)|}{|V(G)|} \leq \alpha - 1+ 2\Delta < (1 +2\eta) k = (1+\varepsilon) k = 2\ell.
        \]
    \end{proof}
    Since $G$ can be made arbitrarily large by choosing appropriately the value of $\beta$, the two claims conclude the proof.
\end{proof}

\begin{remark}
    In the above proof, we can take $\Delta = \Theta\left(\frac{1}{\varepsilon}\ln\frac{1}{\varepsilon}\right)$ and $k = \Theta\left(\frac{1}{\varepsilon^2}\ln\frac{1}{\varepsilon}\right)$.
\end{remark}

\section{Conclusion}\label{sec:ccl}

We disproved ~\autoref{conj:bold} in a strong form, but only for very large $k$. It would be interesting to establish the smallest $k$ for which~\autoref{conj:bold} strays from the truth, especially if it turns out to be already at $k=3$. It also seems reasonable to believe that the stronger form, requiring a cut to be not only $(k-1)$-colourable but in fact $(k-2)$-degenerate, would break down earlier, maybe indeed for $k=3$. 

\bibliographystyle{alphaurl}
\bibliography{coolnew}

\end{document}